\documentclass{amsart}

\usepackage{upgreek}
\usepackage{macros}
\standardsettings
\colorcommentstrue

\usepackage[top=1in, left=1in, right=1in, bottom=1in]{geometry}
\usepackage{epsfig}
\DeclareMathAlphabet{\mathpzc}{OT1}{pzc}{m}{it}
\newcommand{\setof}[1]{\left\{#1\right\}}

\renewcommand{\ww}{\omega^\omega}
\newcommand{\bG}{\boldsymbol{\Gamma}}
\newcommand{\bGd}{\check{\boldsymbol{\Gamma}}}
\newcommand{\ad}{\mathsf{AD}}
\newcommand{\concat}{{}^\smallfrown}
\renewcommand{\S}{\sectionsymbol}

\newcommand{\hr}{\frac{1}{2}\R}
\newcommand{\adr}{\ad_{\R}}
\newcommand{\adhr}{\ad_{\frac{1}{2}\R}}
\newcommand{\adp}{\ad^+}
\newcommand{\unif}{\text{Unif}}

\newcommand{\dom}{\text{dom}}
\newcommand{\res}{\restriction}

\newcommand{\ac}{\mathsf{AC}}
\newcommand{\hR}{\hr}
\newcommand{\lr}{L(\mathbb{R})}
\newcommand{\zf}{\mathsf{ZF}}
\newcommand{\zfc}{\mathsf{ZFC}}

\newcommand{\bP}{\boldsymbol{\Pi}}
\newcommand{\cC}{\mathcal{C}}
\newcommand{\inter}{\text{int}}
\newcommand{\dc}{\mathsf{DC}}
\newcommand{\bS}{\boldsymbol{\Sigma}}

\draftfalse

\begin{document}

\authorlogan\authorlior\authorsteve

\title{Determinacy of Schmidt's Game and Other Intersection Games}

\begin{abstract}
Schmidt's game, and other similar intersection games have played an important
role in recent years in applications to number theory, dynamics, and
Diophantine approximation theory. These games are real games, that is, games
in which the players make moves from a complete separable metric space. The determinacy
of these games trivially follows from the axiom of determinacy for
real games, $\adr$, which is a much stronger axiom than that asserting  all
integer games are determined, $\ad$. One of our main results 
is a general theorem which under the hypothesis $\ad$ implies the
determinacy of intersection games which have a property allowing
strategies to be simplified. In particular, we show that
Schmidt's $(\alpha,\beta,\rho)$ game on $\R$ is determined from
 $\ad$ alone, but on $\R^n$ for $n \geq  3$ we show that $\ad$ 
does not imply the determinacy of this game. We also prove several other
results specifically related to the determinacy of Schmidt's game. These results highlight
the obstacles in obtaining the determinacy of Schmidt's game from $\ad$. 
\end{abstract}

\maketitle

\section{Introduction} \label{sec:introduction}

In 1966, Schmidt \cite{Schmidt1} introduced a two-player game referred to thereafter as
Schmidt's game. Schmidt invented the game primarily as a tool for
studying certain sets which arise in
number theory and Diophantine approximation theory. Schmidt's game, and other similar games,
have since become an important tool in number theory, dynamics and related areas.

Schmidt's game (defined precisely in Subsection \ref{Schmidtsgamedef}) and related games are real games, that is games in which each player plays
a ``real'' (an element of a {\em Polish space}: a completely metrizable and separable space). 
Questions regarding which player, if any, has a winning strategy in various games
have been systematically studied over the last century. Games in which one of the players has a winning strategy are said to be \emph{determined}.
The existence of winning strategies often have implications in both set theory and applications to other areas.  
In fact, the assumption that certain classes of games are determined can have far-reaching structural consequences.  
One such assumption is the axiom of determinacy, $\ad$, which is the statement that all integer games are determined.  
The axiom of determinacy for real games, $\adr$, would immediately imply the determinacy of Schmidt's game, but it is significantly stronger 
than $\ad$ (see Subsection \ref{ss:games} for a more thorough discussion).
A natural question is what form of determinacy axiom is necessary to obtain the determinacy of
Schmidt's game. In particular, can one obtain the determinacy of this game from
$\ad$, or does one need the full strength of $\adr$?

Consider the case of the Banach-Mazur game on a Polish space $(X,d)$ with target set $T \subseteq X$.
Here the players I and II at each turn $n$  play a real which codes
a closed ball $B(x_n,\rho_n)=\{ y\in X\colon d(x_n,y)\leq \rho_n\}$. The only ``rule''
of the game is that  the players must play a decreasing sequence of closed balls (that is,
the first player to violate this rule loses). If both player follows the rule,
then II wins iff $\bigcap_n B(x_n,\rho_n) \cap T \neq \emptyset$. Although this is a real game,
this game is determined for any $T \subseteq X$ just from $\ad$.
This follows from the easy fact that the Banach-Mazur game is equivalent to the integer game
in which both players play closed balls with ``rational centers'' (i.e., from
a fixed countable dense set) and rational radii.

For Schmidt's game on a Polish space $(X,d)$ with target set $T\subseteq X$, we have
in addition fixed parameters $\alpha,\beta\in (0,1)$. In this game I's first move is
a closed ball $B(x_0,\rho_0)$ as in the Banach-Mazur game. In subsequent moves, the
players play a decreasing sequence of closed balls as in the Banach-Mazur game, but with
a restriction of the radii. Namely, II must shrink the previous radius by a factor of $\alpha$,
and I must shrink the previous radius by $\beta$. So, at move ${2n}$, I plays a closed ball
of radius $\rho_{2n}=(\alpha \beta)^n \rho_0$, and at move ${2n+1}$, II plays a closed ball of radius
$\rho_{2n+1}=\alpha (\alpha \beta)^n \rho_0$. As with the Banach-Mazur game, if both players follow these rules,
then II wins iff $x \in T$ where $\{ x\}=\bigcap_n B(x_n,\rho_n)$.
We call this game the $(\alpha,\beta)$ Schmidt's game for $T$.
A variation of Schmidt's game, first introduced by Akhunzhanov in \cite{Akh},
has an additional rule that the initial radius $\rho_0=\rho$ of I's first move is fixed in advance. We call this the
$(\alpha,\beta,\rho)$ Schmidt's game for $T$. 
In all practical applications of the game
we are aware of, the difference between these two versions is immaterial. 
However, in general, these games are not literally 
equivalent, as the following simple example demonstrates.

\begin{example}
Consider $\mathbb{R}$ with the usual metric and let the target set for
II be $T=(-\infty, -1] \cup [1, \infty) \cup \mathbb{Q}$.  Notice that
  this set is dense.  It is easy to see that if $\rho \geq 2$ and
  $\alpha \leq \frac14$ then for any $\beta$, II wins the $(\alpha,
  \beta, \rho)$-game, simply by maximizing the distance from the
  center of her first move to the origin.  But if I is allowed to
  choose any starting radius and $\beta < \frac12$, then he is allowed
  to play, for instance, $(0, \frac12)$, and then on subsequent moves,
  simply avoid each rational one at a time, so that in fact I wins
  the $(\alpha, \beta)$-game.
\end{example}

In the case of Schmidt's game (either variation) it is not immediately clear
that the game is equivalent to an integer game, and thus it is not clear that $\ad$
suffices for the determinacy of these games. Our main results have implications regarding the determinacy
of Schmidt's game.

Another class of games which is similar in spirit to Schmidt's game are the so-called Banach games
whose determinacy has been investigated by Becker and Freiling \cite{Becker} \cite{Freiling} (with an important
result being obtained by Martin). Work of these authors has shown that the determinacy of these games
follows from (and is, in fact, equivalent to) $\ad$. Methods similar to those used by Becker, Freiling, and Martin
are instrumental in the proofs of our results as well.

In \S\ref{sec:background} we introduce notation and give some relevant background in the theory of games, descriptive set theory, and the
history of Schmidt's game in particular.

In \S\ref{sec:mr} we prove our main results, including those regarding the determinacy of Schmidt's game.
We prove general results, Theorems~\ref{hrthm}, \ref{detthm}, which give some conditions under which
certain real games are determined under $\ad$ alone. Roughly speaking, these results state that
``intersection'' games which admit strategies which are simple enough to be ``coded by a real,'' in a sense to
made precise, are determined from $\ad$. Schmidt's game, Banach-Mazur games, and other similar games
are intersection games. The simple strategy condition, however, depends on the specific game.
For Schmidt's $(\alpha,\beta,\rho)$ game on $\R$, we show the simple strategy condition is met,
and so this game is determined
from $\ad$. Moreover, for the
$(\alpha,\beta)$ Schmidt's game on $\R$, $\ad$ implies that either player I has a winning strategy
or else for every $\rho$, II has a winning strategy in the $(\alpha,\beta,\rho)$ game
(this does not immediately give a strategy for II in the $(\alpha,\beta)$ game from $\ad$,
as we are unable in the second case to choose, as a function of $\rho$,
a winning strategy for II in the $(\alpha,\beta,\rho)$ game).
For $\R^n$, $n \geq 2$, the simple strategy condition is not met. In fact, for $n \geq 3$
we show that the determinacy of Schmidt's $(\alpha,\beta,\rho)$ games does not follow from $\ad$.
For $n=2$, we do not know if
$\ad$ suffices to get the determinacy of Schmidt's game.

In \S\ref{sec:or} we prove two other results related to the determinacy of Schmidt's game in particular.
First, we show assuming $\ad$ that in any Polish space $(X,d)$, any $p \in (0,1)$, and any
$T \subseteq X$, there is at most one value of $(\alpha,\beta)\in (0,1)^2$ with $\alpha\beta=p$
such that the $(\alpha,\beta)$ Schmidt's game for $T$ is not determined. 
Second, we show assuming $\ad$ that for a general Polish space $(X,d)$ and any target set $T\subseteq X$,
the ``non-tangent'' version of Schmidt's $(\alpha,\beta,\rho)$
game is determined. This game is just like Schmidt's game except we require
each player to play a ``non-tangent ball,'' that is, $d(x_n,x_{n+1}) < \rho_n-\rho_{n+1}$.  These results help to illuminate the 
obstacles in analyzing the determinacy of Schmidt's game.

Finally in \S\ref{sec:questions} we list several open questions which are left unanswered by
our results. We feel that the results and questions of the current paper show an interesting
interplay between determinacy axioms and the combinatorics of Schmidt's game.

\section{Background}   \label{sec:background}

In this section we fix the notation we use to describe the games we will be considering,
both for general games and specifically for Schmidt's game. We recall some facts about the forms of determinacy we will
be considering, some necessary background in descriptive set theory to state and prove our theorems, and we explain some of the history and significance of Schmidt's game. 

Throughout we let $\omega=\N=\{ 0,1,2,\dots\}$ denote the set of natural numbers.
We let $\R$ denote the set of real numbers (here we mean the elements of the standard real
line, not the Baire space $\omega^\omega$ as is frequently customary in
descriptive set theory).

\subsection{Games} \label{ss:games}

Let $X$ be a non-empty set. Let $X^{<\omega}$ and $X^\omega$ denote respectively the set of
finite and infinite sequences from $X$. For $s \in X^{<\omega}$ we let
$|s|$ denote the length of $s$. If $s,t \in X^{<\omega}$ we write
$s \leq t$ if $s$ is an initial segment of $t$, that is, $t\res |s|=s$. 
If $s,t \in X^{<\omega}$, we let $s \concat t$ denote the concatenation of $s$ and $t$.

We call $R\subseteq X^{<\omega}$ a {\em tree on $X$} if it is closed under initial segments,
that is, if $t \in R$ and $s\leq t$, then $s \in R$. We can view $R$ as the set of
{\em rules} for a game. That is, each played must move at each turn so that the
finite sequence produced stays in $R$ (the first player to violate this ``rule''
loses the game). If $\vec{x}=(x_0,x_1,\dots) \in X^\omega$, we say $\vec{x}$
has followed the rules if $\vec{x}\res n \in R$ for all $n$. We let $[R]$
denote the set of all $\vec x \in X^{\omega}$ such that $\vec x \res n \in R$ for all $n$
(i.e., $\vec x$ has followed the rules). We also refer to $[R$] as the set of {\em branches}
through $R$.
We likewise say
$s \in X^{<\omega}$ has followed the rules just to mean $s \in R$.

Fix a set $B \subseteq X^\omega$, which we call the {\em target set},
and let $R \subseteq X^{<\omega}$ be a rule set (i.e., a tree  on $X$). 
The game $G(B,R)$ on the set $X$  is defined as follows. I and II alternate
playing elements $x_i \in X$. So, I plays $x_0,x_2,\dots$, while II
plays $x_1,x_3,\dots$. This produces the {\em run} of the game
$\vec x=(x_0,x_1,\dots)$. The first player, if any, to violate the rules $R$
loses the run $\vec x$ of the game. If both players follow the rules
(i.e., $\vec x\in [R]$), then we declare I to have won the run iff $\vec x \in B$
(otherwise we say II has won the run). 
Oftentimes,  in defining a game the set of rules $R$ is defined implicitly
by giving  requirements on each players' moves.
If there are no rules, i.e., $R=X^{<\omega}$, then we write
$G(B)$ for $G(B,R)$. 
Also, it is frequently convenient
to define the game by describing the payoff set for II instead of I.
This, of course, is formally just replacing $B$ with $X^\omega-B$.

A {\em strategy} for I in a game on the set $X$ is a function
$\sigma \colon \bigcup_{n\in \omega} X^{2n} \to X$. A strategy for II
is a function $\tau \colon \bigcup_{n \in \omega} X^{2n+1}\to X$.
We say $\sigma$ follows the rule set $R$ is whenever $s \in R$ of
even length, than $s \concat \sigma(s)\in R$. We likewise define
the notion of a strategy $\tau$ for II to follow the rules.
We say $\vec x \in X^{\omega}$ follows
the strategy $\sigma$ for I if for all $n \in \omega$,
$x_{2n}=\sigma (\vec x\res 2n)$, and similarly define
the notion of $\vec x$ following the strategy $\tau$ for
II. We also extend this terminology in the obvious way
to say an $s \in X^{<\omega}$ has followed $\sigma$ (or $\tau$).
Finally, we say a strategy $\sigma$ for I is a {\em winning strategy}
for I in the game $G(B,R)$ if $\sigma$ follows the rules $R$
and for all $\vec x\in [R]$ which follows $\sigma$ we have $\vec x\in B$,
that is, player I has won the run $\vec x$. We likewise define the notion
of $\tau$ being a winning strategy for II.

If $\sigma$ is a strategy for I, and $\vec z=(x_1,x_3,\dots)$ is a sequence of moves
for II, we write $\sigma* \vec{z} $ to denote the corresponding
run $(x_0,x_1,x_2,x_3,\dots)$ where $x_{2n}=\sigma( x\res 2n)$.
We likewise define $\tau * \vec z$ for $\tau$ a strategy for II and
$\vec z=(x_0,x_2,\dots)$ a sequence of moves for I. If $\sigma, \tau$
are strategies for I and II respectively, then we let $\sigma*\tau$
denote the run $(x_0,x_1,\dots)$ where $x_{2n}=\sigma( x\res 2n)$
and $x_{2n+1}=\tau( x\res 2n+1)$ for all $n$.

We say the game $G(B,R)$ on $X$ is {\em determined} if one of the players
has a winning strategy. The {\em axiom of determinacy} for games on
$X$, denoted $\ad_X$ is the assertion that all games on the set $X$
are determined. Axioms of this kind were first introduced by Mycielski 
and Steinhaus. We let $\ad$ denote $\ad_\omega$, that is, the assertion
all two-player integer games are determined. Also important for the current paper
is the axiom $\adr$, the assertion that all real games are determined. 
Both $\ad$ and $\adr$ play an important role in modern descriptive set theory.
although both axioms contradict the axiom of choice, $\ac$, and thus are not
adopted as axioms for the true universe $V$ of set theory, they
play a critical role in developing the theory of natural models such as
$\lr$ containing ``definable'' sets of reals.  It is known that $\adr$
is a much stronger assertion than $\ad$ (see Theorem 4.4 of \cite{solovay}).

Sitting between $\ad$ and $\adr$ is the determinacy of another class of games called
{\em $\hR$} games, in which one of the players plays reals and the other
plays integers. The proof of one of our theorems will require the use of
$\hR$ games. The axiom $\adhr$ that all $\hR$ games are determined is known
to be equivalent to $\adr$ ($\adhr$ immediately implies $\unif$, see Theorem~\ref{martinwoodin} below). However, $\ad$ suffices to obtain the determinacy
of $\hR$ games with Suslin, co-Suslin payoff (a result of Woodin, see \cite{kechrishr}).
We define these terms more precisely in \S\ref{sec:mr}. As in \cite{Becker},
this fact will play an important role in one of our theorems.

One of the central result in the theory of games is the result of Martin \cite{Martin_determinacy}
that all Borel games on any set $X$ are determined in $\zfc$.
By ``Borel'' here we are referring to the topology on $X^\omega$ given by the
product of the discrete topologies on $X$. 
In fact, in just $\zf$
we have that all Borel games (on any set $X$) are {\em quasi-determined}
(see \cite{Moschovakis} for the definition of quasi-strategy and proof of the extension of Martin's  result
to quasi-strategies in $\zf$, which is due to Hurkens  and Neeman).

\begin{theorem}[Martin, Hurkens and Neeman for quasi-strategies]
\label{theoremboreldeterminacy}
Let $X$ be a nonempty set, and let $B\subseteq X^\omega$ be a Borel set, and $R\subseteq X^{<\omega}$
a rule set $R$ (a tree). 
Then the game $G(B,R)$ is determined (assuming $\zfc$, or quasi-determined just assuming $\zf$).
\end{theorem}

As we mentioned above, $\ad$ contradicts $\ac$. In fact, games played for particular types of
``pathological'' sets constructed using $\ac$ are frequently not determined.
For example, the following result is well-known (e.g. \cite[p. 137, paragraph 8]{Kechris}):

\begin{proposition}
\label{propositiongalestewartundetermined}
Let $B \subseteq \omega^\omega$ be a Bernstein set (i.e., neither the set nor its complement
contains a perfect set). Then the game $G(B)$ is not determined.
\end{proposition}

\subsection{Determinacy and Pointclasses} \label{detpc}

We briefly review some of the terminology and results related to the
determinacy of games and some associated notions concerning pointclasses which
we will need for the proofs of some of our results.

We have introduced above the axioms $\ad$, $\adhr$, and $\adr$ which assert the determinacy
of integer games, half-real games, and real games respectively. We trivially have
$\adr \Rightarrow \adhr \Rightarrow \ad$. All three of these axioms contradict $\ac$,
the axiom of choice. They are consistent, however, with $\dc$, the axiom of
dependent choice, which asserts that if $T$ is a non-empty {\em pruned} tree
(i.e., if $(x_0,\dots,x_n)\in T$ then $\exists x_{n+1}\ (x_0,\dots,x_n,x_{n+1})\in T$)
then there is a branch $f$ through $T$ (i.e., $\forall n\ (f(0),\dots,f(n))\in T$).
$\dc$ is a slight strengthening of the axiom of countable choice. On the one hand,
$\dc$ holds in the minimal model $\lr$ of $\ad$, while on the other hand even
$\adr$ does not imply $\dc$. Throughout this paper, our background theory is $\zf+\dc$.

The axiom $\adr$ is strictly stronger than $\ad$ (see \cite{solovay}), and in fact
it is known that $\adr$ is equivalent to $\ad+\unif$, where $\unif$
is the axiom that every $R\subseteq \R\times \R$ has a {\em uniformization},
that is, a function $f \colon\dom(R)\to \R$ such that $(x,f(x))\in R$
for all $x \in \dom(R)$. This equivalence will be important for
our argument in Theorem~\ref{thm:r3} that $\ad$ does not suffice for the
determinacy of Schmidt's game in $\R^n$ for $n \geq 3$.
The notion of uniformization is closely connected with the descriptive set
theoretic notion of a {\em scale}. If a set $R\subseteq X\times Y$ (where
$X$, $Y$ are Polish spaces) has a scale, then it has a uniformization. The only property of scales which we use is the existence of uniformizations, so we will not give the definition, which is rather technical, here.

A (boldface) {\em pointclass} $\bG$ is a collection
of subsets of Polish spaces closed under continuous preimages, that is, if
$f \colon X\to Y$ is continuous and $A\subseteq Y$ is in $\bG$, then
$f^{-1}(A)$ is also in $\bG$. We say $\bG$ is selfdual if $\bG=\bGd$ where
$\bGd=\{ X-A\colon A\in \bG\}$ is the dual pointclass of $\bG$. We say
$\bG$ is non-selfdual if $\bG\neq \bGd$. A set $U \subseteq \ww \times X$
is {\em universal} for the $\bG$ subsets of $X$ if $U\in \bG$ and for every
$A\subseteq X$ with $A\in \bG$ there is an $x \in \ww$ with $A=U_x=\{ y\colon (x,y)\in U\}$.
It is a consequence of $\ad$ that every non-selfdual pointclass has a universal set.

For $\kappa$ an ordinal number we say a set
$A \subseteq \ww$ is $\kappa$-Suslin if there is a tree $T$ on $\omega \times \kappa$
such that $A=p[T]$, where $p[T]=\{ x \in \ww \colon \exists f \in \kappa^\omega\
(x,f)\in [T]\}$ denotes the projection of the body of the tree $T$. We say $A$
is Suslin if it is $\kappa$-Suslin for some $\kappa$. We say $A$ is
co-Suslin if $\ww \setminus A$ is Suslin. For a general Polish space $X$,
we say $A \subseteq X$ is Suslin if for some continuous surjection
$\varphi \colon \ww \to X$ we have that $\varphi^{-1}(A)$ is Suslin
(this does not depend on the choice of $\varphi$). Scales are essentially
the same thing as Suslin representations, in particular a set $A\subseteq Y$
is Suslin iff it has a scale, thus relations which are Suslin have uniformizations.
If $\bG$ is a pointclass, then we say a set $A$ is {\em projective over} $\bG$
if it is in the smallest pointclass $\bG'$ containing $\bG$ and closed under
complements and existential and universal quantification over $\R$.
Assuming $\ad$, if $\bG$ is contained in the class of Suslin, co-Suslin sets, then every set projective over
$\bG$ is also Suslin and co-Suslin.  For this result, 
more background
on these general concepts, as well as the precise definitions of scale and the scale property,
the reader can refer to \cite{Moschovakis}.

Results of Martin and Woodin (see \cite{MartinWoodin} and \cite{Martin_ctb})
show that assuming $\ad+\dc$, the axioms
$\adr$, $\unif$, and scales are all equivalent. More precisely we have the following.

\begin{theorem} [Martin, Woodin] \label{martinwoodin}
Assume $\zf+\ad+\dc$. Then the following are equivalent:
\begin{enumerate}
\item
$\adr$
\item
$\unif$
\item
Every $A\subseteq \R$ has a scale.
\end{enumerate}
\end{theorem}

Scales and Suslin representations are also important as it follows from $\ad$
that ordinal games where the payoff set is Susin and co-Suslin (the notion of
Suslin extends naturally to sets $A \subseteq \lambda^\omega$ for $\lambda$ an ordinal number)
are determined (one proof of this is due to Moschovakis, Theorem~2.2 of \cite{Moschovakis_od}, another
due to Steel can be found in the proof of Theorem~2 of
\cite{Steel}). We will not need this result for the current paper.

A strengthening of $\ad$, due to Woodin, is the axiom $\adp$. This axiom has been very useful
as it allows the development of a structural theory which has been used to obtain a number of
results. It is not currently known if $\adp$ is strictly stronger than $\ad$, but
it holds in all the natural models of $\ad$ obtained from large cardinal axioms
(it holds, in particular, in the model $\lr$, so $\adp$ is strictly weaker that $\adr$).
In our Theorem~\ref{thm:r3} we in fact show that $\adp$ does not suffice
to get the determinacy of Schmidt's $(\alpha,\beta,\rho)$ game in $\R^n$ for $n \geq 3$.

\subsection{Schmidt's game} \label{Schmidtsgamedef}

As mentioned in the introduction, Schmidt 
invented the game primarily as a tool for
studying certain sets which arise in
number theory and Diophantine approximation theory. These sets are
often exceptional with respect to both measure and category, i.e., Lebesgue null and meager.
One of the the most significant examples is the following.
Let $\mathbb{Q}$ denote the set of rational numbers. A real number $x$ is said
to be \emph{badly approximable} if there exists a positive constant
$c=c(\alpha )$ such that $\left|x-\frac{p}{q}\right|>\frac{c}{q^2}$ for
all $\frac{p}{q}\in \mathbb{Q}$.
We denote the set of badly approximable numbers by{ \bf{BA}}.
This set plays a major role in Diophantine approximation theory, and is well
known to be both Lebesgue null and meager.
Nonetheless, using his game, Schmidt was
able to prove the following remarkable result:

\begin{theorem}[Schmidt \cite{Schmidt1}]
Let $(f_n)_{n=1}^{\infty}$  be a sequence of $\CC^1$ diffeomorphisms of $\R$. Then the Hausdorff dimension of the set $\bigcap_{n=1}^{\infty}f^{-1}_n({\bf{BA}})$ is $1$. In particular, $\bigcap_{n=1}^{\infty}f^{-1}_n({\bf{BA}})$ is uncountable.
\end{theorem}

Yet another example of the strength of the game is the following.
Let $b\geq 2$ be an integer. A real number $x$ is said to be normal to base $b$ if, for every
$n\in\mathbb{N}$, every block of $n$ digits from $\{0, 1,\dots , b-1\}$ occurs in the base-$b$ expansion of
$x$ with asymptotic frequency $1/b^n$. It is readily seen that the set of numbers normal to no base is both Lebesgue null and meager. Nevertheless, Schmidt used his game to prove:

\begin{theorem}[Schmidt \cite{Schmidt1}]
The Hausdorff dimension of the set of numbers normal to no base is $1$.\\
\end{theorem}

\subsubsection{The game's description}

For the $(\alpha,\beta)$ Schmidt's game on the complete metric space $(X, d)$ with target set $T \subseteq X$, I and II each
play pairs $(x_i, \rho_i)$ in $Y=X \times \R^{>0}$. The $R \subseteq Y^{<\omega}$
of rules is defined by the conditions that $\rho_{i+1}+d(x_i,x_{i+1})\le \rho_{i}$
and $\rho_{i+1}= \begin{cases} \alpha \rho_i & \text{ if } i \text{ is even }\\
\beta \rho_i & \text{ if } i \text{ is odd }\end{cases}$. The rules guarantee that
the closed balls $B(x_i,\rho_i)=\{ x \in \R^n \colon d(x,x_i)\leq \rho_i\}$
are nested. Since the $\rho_i\to 0$, there is a unique point $z \in X$
such that $\{ z\}=\bigcap_i B(x_i,\rho_i)$. For $\vec x\in [R]$, a run of the game
following the rules, we let $f(\vec x)$ be this corresponding point $z$.
The payoff set $B\subseteq Y^\omega$ for player I is $\{ \vec x \in Y^\omega \cap [R]\colon
f(\vec x) \notin T\}$. Formally, when we refer to the $(\alpha,\beta)$ Schmidt's game with
target set $T$, we are referring to the game $G(B,R)$ with these sets $B$ and $R$
just described. The formal definition of Schmidt's $(\alpha,\beta,\rho)$
game with target set $T$ is defined in the obvious analogous manner.

\section{Main Results} \label{sec:mr}

We next prove a general result which states that certain real games are equivalent to $\hr$
games. The essential point is that real games which are intersection games (i.e., 
games where the payoff only depends on the intersection of sets coded by the moves
the players make) with the property that if one of the players has a winning strategy in the real game,
then that player has a strategy ``coded by a real'' (in a precise sense defined below), 
then the game is equivalent to a $\hr$ game. In \cite{Becker} a result attributed to Martin
is presented which showed that the determinacy of a certain class of real games, called Banach games,
follows from $\adhr$, the axiom which asserts the determinacy of $\hr$ games (that is,
games in which one player plays reals, and the other plays integers). In Theorem~\ref{hrthm}
we use ideas similar to Martin's to prove a general result which applies to
intersection games satisfying a ``simple strategy'' hypothesis. Since
many games with applications to number theory and dynamics are intersection games, it seems that in practice
the simple strategy hypothesis is the more significant requirement.

\begin{definition} \label{simple_one_round}
Let $\bG$ be a pointclass. A simple one-round $\bG$ strategy $s$ for the Polish space $X$ is 
a sequence $s=(A_n, y_n)_{n \in \omega}$ where $y_n \in X$, $A_n \in \bG$, and the $A_n$ are a partition 
of $X$. 
A simple $\bG$ strategy $\tau$ for player II is a collection $\{ s_u \}_{u \in \omega^{<\omega}}$
of simple one-round $\bG$ strategies $s_u$. A simple $\bG$ strategy $\sigma$ for player I
is a pair $\sigma=(\bar{y}, \tau)$ where $\bar{y}\in X$ is the first move and
$\tau$ is a simple $\bG$ strategy for player II.

\end{definition}

The idea for a simple one-round strategy is that  if the opponent moves in the set 
$A_n$, then the strategy will respond with $y_n$. Thus there is only ``countably much''
information in the strategy; it is coded by a real in a simple manner.
If $s=(A_n,y_n)$ is a simple one-round strategy,
we will write $s(n)=y_n$ and also $s(x)=y_n$ for any $x \in A_n$. 
A general 
simple strategy produces after each round a new simple one-round strategy to follow
in the next round. For example, suppose $\sigma$ is a simple strategy for I. 
$\sigma$ gives a first move $x_0=\bar{y}$ and a simple one-round strategy 
$s_\emptyset$. If II plays $x_1$, then $x_2=\sigma(x_0,x_1)=s_\emptyset(x_1)=
$the unique $y_{n_0}$ such that $x_1 \in A_{n_0}$ where $s_\emptyset=(A_n,y_n)$. 
If II then plays $x_3$, then $\sigma$ responds with $s_{n_0}(x_3)$. The play by
$\sigma$ continues in this manner. Formally, a general simple strategy is
a sequence $(s_u)_{u \in \omega^{<\omega}}$ of simple one-round strategies,
indexed by $u \in \omega^{<\omega}$.

If $\bG$ is a pointclass with a universal set $U\subseteq \ww\times X$,
then we may use $U$ to code simple one-round $\bG$ strategies. Namely,
the simple one-round $\bG$ strategy $s=(A_n,y_n)$ is coded by $z \in \ww$
if $z$ codes a sequence $(z)_n \in \ww$ and $U_{(z)_{2n}} =A_n$ and
$(z)_{2n+1}$ codes the response $y_n\in X$ in some reasonable manner
(e.g., via a continuous surjection from $\ww$ to $X$, the exact details are
unimportant).

\begin{remark}
For the remainder of this section, $X$ and $Y$ will denote Polish spaces.
\end{remark}

\begin{definition}
Let $R \subseteq X^{<\omega}$ be a tree on $X$ which we identify as a {\em set of rules}
for a game on $X$. We say a simple one-round
$\bG$ strategy $s$ {\em follows the rules} $R$ at position $p \in R$ if 
for any $x \in X$, if $p \concat x \in R$, then $p \concat x \concat s(x)\in R$. 
\end{definition}

\begin{definition}
Let $R \subseteq X^{<\omega}$ be a set of rules for a real game. 
Suppose $p \in X^{<\omega}$ is a position in $R$. 
Suppose $f \colon X \to X$ is such that for all $x \in X$, if $p\concat x\in R$,
then $p\concat x \concat f(x) \in R$ (i.e., $f$ is a one-round strategy which follows the rules at $p$).
A {\em simplification} of $f$ at $p$ is simple one-round strategy $s=(A_n,y_n)$ such that 

\begin{enumerate}
%% \item
%% For every $n$, there is an $x \in A_n$ such that $p\concat x \in R$ and 
%% $f(x)=y_n$. 
\item
For every $x$ in any $A_n$, if $p \concat x \in R$, then $p \concat x \concat y_n \in R$.
\item
For every $n$, if there is an $x \in A_n$ such that $p\concat x \in R$, 
then there is an $x' \in A_n$ with $p\concat x'\in R$ and $f(x')=y_n$.
\end{enumerate}
We say $\tau$ is a $\bG$ simplification of $f$ if all of the set $A_n$
are in $\bG$.

\end{definition}

\begin{definition}
We say a tree $R\subseteq X^{<\omega}$ is {\em positional} if for all $p,q \in R$
of the same length
and $x\in X$, if $p \concat x$, $q \concat x$ are both in $R$ 
then for all $r \in X^{<\omega}$, $p\concat x\concat r \in R$ iff 
$q \concat x \concat r \in R$. 
\end{definition}

\begin{theorem}[$\zf+\dc$] \label{hrthm}
Let $\bG$ be a pointclass with a universal set with $\bG$ contained within the Suslin, co-Suslin sets. 
Suppose $B\subseteq X^\omega$ and $R \subseteq X^{<\omega}$ is a positional tree,
and suppose both $B$ and $R$ are in $\bG$. 
Let $G=G(B,R)$ be the real game on $X$ with payoff $B$ and rules $R$. Suppose the following two conditions
on $G$ hold:
\begin{enumerate}
\item (intersection condition)
For any $\vec{x},\vec{y}\in [R]$, if $x(2k)=y(2k)$ for all $k$, then 
$\vec{x}\in B$ iff $\vec{y}\in B$. 
\item (simple one-round strategy condition)
If $p\in R$ has odd length, and $f\colon X \to X$ is a rule following 
one-round strategy at $p$, then there is a $\bG$-simplification of $f$ at $p$. 
\end{enumerate}
Then $G$ is equivalent to a Suslin, co-Suslin $\hr$ game $G^*$ in the sense that if I (or II)
has a winning strategy in $G^*$, then I (or II) has a winning strategy in $G$.
\end{theorem}

\begin{proof}
Consider the game $G^*$ where I plays pairs $(x_{2k},s_{2k})$ and II plays 
integers $n_{2k+1}$. The rules $R^*$ of $G^*$ are that I must play at each round
a real coding $s_{2k}$ which is a simple one-round $\bG$ strategy which follows the rules $R$ 
relative to a position $p\concat x_{2k}$ for any $p$ of length $2k$ 
(this does not depend on the particular choice of $p$ as $R$ is positional).
I must also play such that $x_{2k}= s_{2k-2}(n_{2k-1})$. II must play 
each $n_{2k+1}$ so that there is a legal move $x_{2k+1} \in A^{s_{2k}}_{n_{2k+1}}$ 
with $p\concat x_{2k} \concat x_{2k+1} \in R$ (for any $p$ of length $2k$).

If I and II have followed the rules, to produce $x_{2k}, s_{2k}$ and $n_{2k+1}$, the payoff
condition for $G^*$ is as follows. Since II has followed the rules,
there is a sequence $x_{2k+1}$ such that the play $(x_0,x_1,\dots)\in [R]$. 
I then wins the run of $G^*$ iff $(x_0,x_1,\dots)\in B$. Note that by the intersection
condition, this is independent of the particular choice of the $x_{2k+1}$. 

From the definition,  $G^*$ is a Suslin, co-Suslin game.

We show that $G^*$ is equivalent to $G$. Suppose first that I wins $G^*$ by 
$\sigma^*$. Then $\sigma^*$ easily gives a strategy $\Sigma$ for $G$. For example,
let $\sigma^*(\emptyset)= (x_0, s_0)$. Then $\Sigma(\emptyset)=x_0$. 
If II plays $x_1$, then let $n_1$ be such that $x_1 \in A^{s_0}_{n_1}$. Then 
$\Sigma(x_0,x_1)=s_0(n_1)$. Continuing in this manner defines $\Sigma$. 
If $(x_0,x_1,\dots)$ is a run of $\Sigma$, then there is a corresponding
run $((x_0,s_0), n_1, \dots)$ of $\sigma^*$. As each $s_{2k}$ follows the rules
$R$, then as long as II's moves follow the rules $R$, I's moves by $\Sigma$
also follow the rules $R$.  If II has followed the rules $R$ in the run of $G$,
then the run $((x_0,s_0), n_1, \dots)$ of $\sigma^*$ has followed the rules for $G^*$
(II has followed the rules of $G^*$ since for each $n_{2k+1}$,
$x_{2k+1}$ witnesses that $n_{2k+1}$ is a legal move). Since $\sigma^*$ is 
winning for $G^*$, the sequence $(x_0,x'_1,x_2,x'_3,\dots)\in B\cap [R]$ for some 
$x'_{2k+1}$. By the intersection condition, $(x_0,x_1,x_2,x_3,\dots)\in B$.

Assume now that II has winning strategy $\tau'$ in $G^*$. We first note that
there is winning strategy $\tau^*$ for II in $G^*$ such that $\tau^*$ is projective
over $\bG$. To see this, first note that the payoff set for $G^*$ is projective
over $\bG$ as both $B$ and $R$ are in $\bG$. Also, there is a scaled pointclass
$\bG'$, projective over $\bG$, which contains the payoff set for II in $G^*$.
By a result of Woodin in \cite{kechrishr} (since II is playing the integer moves in $G^*$)
there is a winning strategy $\tau^*$ which is projective over $\bG'$, and thus projective over $\bG$. 
For the rest of the proof we fix a winning strategy $\tau^*$ for II in $G^*$
which is projective over $\bG$.

We define a strategy
$\Sigma$ for II in $G$. Consider the first round of $G$. Suppose I moves with
$x_0$ in $G$. We may assume that $(x_0)\in R$. 

\begin{claim}
There is an $x_1$  with $(x_0,x_1)\in R$ such that for all $x_2$ with
$(x_0,x_1,x_2)\in R$, there is a simple one-round $\bG$ strategy $s_0$
which follows the rules $R$ from position $x_0$ (so $(x_0,s_0)$ is a legal move for I
in $G^*$) such that if $n_1=\tau^*(x_0,s_0)$ then $x_1\in A^{s_0}_{n_1}$
and $x_2=s_0(x_1)$. 

\end{claim}

\begin{subproof}
Suppose not, then for every $x_1$ with $(x_0,x_1)\in R$ there is an $x_2$ with
$(x_0,x_1,x_2)\in R$ which witnesses the failure of the claim. 
Define the relation $S(x_1,x_2)$ to hold iff $(x_0,x_1)\notin R$ or
$(x_0,x_1,x_2)\in R$ and the claim fails, that is, for every simple
one-round $\bG$ strategy $s$ which follows $R$, if we let $n_1=\tau^*(x_0,s)$,
then either $x_1\notin A^{s}_{n_1}$ or $x_2\neq s(x_1)$. 
Since $\tau^*$, $B$, $R$ are projective over $\bG$, so is the relation $S$. 
By assumption, $\dom(S)=\R$. Since $S$ is projective over $\bG$, it is within the scaled
pointclasses, and thus there is a uniformization $f$ for $S$. Note that $f$ follows the
rules $R$. By the simple one-round strategy hypothesis of Theorem~\ref{hrthm},
there is a $\bG$-simplification $s_0$ of $f$. Let $n_1=
\tau^*(x_0,s_0)$. Since $\tau^*$ follows the rules $R^*$ for II, there is an $x_1 \in A^{s_0}_{n_1}$
such that $(x_0,x_1)\in R$. Since $s_0$ is a simplification of $f$,
there is an $x'_1$ with $(x_0,x'_1)\in R$ and $f(x'_1)=s_0(n_1)$. Let $x_2=
f(x'_1)$. From the definition of $S$ we have that $(x_0,x'_1,x_2)\in R$.
Since $S(x'_1,x_2)$, there does not exist an $s$ (following the rules)
such that $(x'_1 \in A^{s}_{n_1}
~\text{and}~ x_2=s(x'_1))$ where $n_1=\tau^*(x_0,s)$. But on the other hand, the $s_0$ we have produced
does have this property. This proves the claim.

\end{subproof}

Now that we've proved this claim, we can attempt to define the strategy $\Sigma$.  
We would like to have $\Sigma(x_0)$ be any $x_1$ as in the claim. Now since the relation
$A(x_0,x_1)$ which says that $x_1$ satisfies the claim relative to $x_0$ is projective
over $\bG$, we can uniformize it to produce the first round $x_1(x_0)$ of the strategy $\Sigma$.

Suppose I now moves $x_2$ in $G$. For each such $x_2$ such that $(x_0,x_1,x_2)\in R$,
there is a rule-following simple one-round $\bG$ strategy $s_0$ as in the claim
for $x_1$ and $x_2$. The relation $A'(x_0,x_2,s_0)$ which says that $s_0$ satisfies the claim for
$x_1=x_1(x_0)$,
$x_2$ is projective over $\bG$ and so has a uniformization $g(x_0,x_2)$.
In the $G^*$ game we have I play $(x_0, g(x_0,x_2))$. Note that $n_1=\tau^*(x_0,s_0)$
is such that $x_1 \in A^{s_0}_{n_1}$, and $x_2=s_0(x_1)$.

This completes the definition of the first round of $\Sigma$, and the proof
that a one-round play according to $\Sigma$ has a one-round simulation according
to $\tau^*$, which will guarantee that $\Sigma$ wins.
The definition of $\Sigma$ for the general round is defined in exactly the same way,
using $\dc$ to continue. 
The above argument also shows that a run of $G$ following $\Sigma$
has a corresponding run of $G^*$ following $\tau^*$. If I has followed the rules
of $G$, then I has followed the rules of $G^*$ in the associated run. Since $\tau^*$
is winning for II in $G^*$, there is no sequence sequence $x'_{2k+1}$
of moves for II such that $(x_0,x'_1,x_2,x'_3,\dots)\in B\cap [R]$. In particular,
$(x_0,x_1,x_2,x_3,\dots) \notin B$ (since $(x_0,x_1,\dots)\in [R]$). Thus, II
has won the run of $G$ following $\Sigma$.

\end{proof}

If $G$ is a real game on the Polish space $X$
with rule set $R$, we say that $G$ is an {\em intersection game}
if it satisfies the intersection condition of Theorem~\ref{hrthm}.
This is equivalent to saying that there is a
function $f\colon X^\omega \to Y$ for some Polish space $Y$
such that $f(\vec x)=f(\vec y)$ if $x(2k)=y(2k)$
for all $k$, and the payoff set for $G$ is of the form $f^{-1}(T)$ for some $T\subseteq Y$.
In many examples, the rules $R$ require the players to play decreasing closed sets with diameters
going to $0$ in some Polish space, and the function $f$ is simply giving the unique point 
of intersection of these sets. If we have a fixed rule set $R$ and a fixed
function $f$, the {\em class of games} $G_{R,f}$ associated to $R$ and $f$ is the collection
of games with rules $R$ and payoffs of the form $f^{-1}(T)$ for $T\subseteq Y$.
Thus, we allow the payoff set $T$ to vary, but the set of rules $R$ and the ``intersection function'' $f$
are fixed. In practice, $R$ and $f$ are usually simple, such as Borel relations/functions.

\begin{theorem}[$\ad$] \label{detthm}
Suppose $\bG$ is a non-selfdual  pointclass within the Suslin, co-Suslin sets
and $G_{R,f}$ is a class of intersection games on the Polish space $X$ with $R$, $f \in \bG$,
and $R$ is positional (as above $f \colon X^\omega\to Y$, where $Y$ is a Polish space).
Suppose that for every $T\subseteq Y$ which is Suslin and co-Suslin, if player
I or II was a winning strategy in $G_{R,f}(T)$, then that player has a
winning simple $\bG$-strategy. Then for every $T\subseteq Y$, the game
$G_{R,f}(T)$ is determined.
\end{theorem}

\begin{proof}
Fix the rule set $R$ and function $f$ in $\bG$. 
Let $T \subseteq Y$, we show the real game 
$G_{R,f}(T)$ is determined. Following Becker, we consider the integer game $G$ where I and II play out
reals $x$ and $y$ which code trees (indexed by $\omega^{<\omega}$) of simple one-round $\bG$ strategies
The winning condition for II as follows.
If exactly one of $x$, $y$  fails to be a simple $\bG$-strategy, then that player loses.
If both fail to code simple $\bG$-strategies, then II wins. If 
$x$ codes a simple $\bG$-strategy $\sigma_x$ and $y$ codes a simple $\bG$-strategy $\tau_y$,
then II wins iff $\sigma_x*\tau_y \in G_{R,f}(T)$, where $\sigma*\tau$ denotes the unique sequence of reals
obtained by playing $\sigma$ and $\tau$ against each other. From $\ad$, the game $G$ is determined.
Without loss of generality we may assume that II has a winning strategy $w$ for $G$. 
Let $S_1\subseteq \ww$ be the set of $z$ such that $z$ codes a simple $\bG$-strategy for player I
which follows the rules $R$. 
Likewise, $S_2$ is the set of $z$ coding rule following $\bG$-strategies $\tau_z$ for II.
Note that $S_1$, $S_2$ are projective over $\bG$. 
Let 
\[
A=\{ \vec{y} \in X^\omega \colon \exists z \in S_1\ \vec y= \sigma_z * \tau_{w(z)} \}.
\]

Since $w$ is a winning strategy for II in $G$, $A\subseteq
X^\omega\setminus G_{R,f}(T)$, so $f(A) \subseteq Y \setminus T$.
Note that $A$ is projective over $\bG$ by the complexity assumption on
$R$ and the fact that $S_1$ is also projective over $\bG$.  We claim
that it suffices to show that II wins the real game $G_{R,f}(Y
\setminus f(A))$. This is because if II wins $G_{R,f}(Y
\setminus f(A))$ with run $\vec y$, i.e.  $\vec y \not \in G_{R,f}(Y
\setminus f(A))$, then $f(\vec y) \in f(A) \subseteq Y \setminus T$,
so $\vec y \not \in G_{R, f}(T)$, thus $\vec y$ is a winning run for II in $G_{R, f}(T)$.

We see that $Y \setminus f(A)$ is projective over $\bG$, and thus
$G_{R, f}(Y \setminus f(A))$ is equivalent to a Suslin, co-Suslin
$\hr$ game by Theorem \ref{hrthm} which is determined (see \cite{kechrishr}), and
so $G_{R, f}(Y \setminus f(A))$ is determined. Now it suffices to
show that I doesn't have a winning strategy in $G_{R, f}(Y \setminus
f(A))$.

Suppose I had a winning strategy for $G_{R, f}(Y \setminus f(A))$. By
hypothesis, I has a winning simple $\bG$-strategy coded by some $z \in
\ww$. Let $\vec y= \sigma_z* \tau_{w(z)}$ (note that $z \in S_1$ and
so $w(z)\in S_2$). Since $\sigma_z$ is a winning strategy for I in
$G_{R, f}(Y \setminus f(A))$, we have $f(\vec y) \in Y \setminus
f(A)$. On the other hand, from the definition of $A$ from $w$ we have
that $f(\vec y) \in f(A)$, a contradiction.
\end{proof}

We next apply Theorem~\ref{detthm} to deduce the determinacy of Schmidt's $(\alpha,\beta,\rho)$
games in $\R$ from $\ad$.

\begin{theorem} [$\ad$]\label{thm:schmidtdet}
For any $\alpha,\beta \in (0,1)$, any $\rho \in \R_{>0}$, and any $T\subseteq \R$,
the $(\alpha,\beta,\rho)$ Schmidt's game with target set $T$ is determined. 
\end{theorem}

\begin{proof}
Let $\bG$ be the pointclass $\bP^1_1$ of co-analytic sets. Let $R$ be the tree
described by the rules of the $(\alpha,\beta,\rho)$ Schmidt's game. $R$
is clearly a closed set and is positional.  The function $f$ of Theorem~\ref{detthm}
is given by $\{ f((x_i,\rho_i)_i) \}= \bigcap_i B(x_i,\rho_i)$. This
clearly satisfies the intersection condition, that is, $G_{R,f}$
is a class of intersection games. Also, $f$ is continuous, so $f\in \bG$.

It remains to verify the simple strategy condition of Theorem~\ref{detthm}.
The argument is essentially symmetric in the players, so we consider the case
of player II. In fact we show that for any $T\subseteq \R$, if II
has a winning strategy for the $(\alpha,\beta,\rho)$ Schmidt's game, then II
has a simple Borel strategy. Fix a winning strategy $\Sigma$
for II in this (real) game. Consider $\Sigma$ restricted to the
first round of the game. For every $z_0 \in \R$, there is a half-open
interval $I_{z_0}$ of the form $[z_0,z_0+\epsilon)$ or $(z_0-\epsilon,z_0]$
such that for any $x_0\in I_{z_0}$, we have that $((x_0,\rho_0),
\Sigma(z_0,\rho_0))\in R$. That is, for any $x_0 \in I_{z_0}$ we have that  $\Sigma$'s
response to $(z_0,\rho_0)$ is still a legal response to the play $(x_0,\rho_0)$. 
Consider the collection $\cC$ of all intervals $I=[z,z+\epsilon)$ or $I=(z-\epsilon,z]$
having this property. So, $\cC$ is a cover of $\R$ by half-open intervals. 
There is a countable subcollection $\cC' \subseteq \cC$ which covers $\R$.
To see this, first get a countable $\cC_0\subseteq \cC$ such that
$\cup \cC_0 \supseteq \bigcup_{I\in \cC} \inter(I)$.
The set $\R\setminus\bigcup_{I\in \cC} \inter(I)$ must be countable, and so
adding countably many sets of $\cC$ to $\cC_0$ will get $\cC'$ as desired. 
Let $\cC'= \{ I_{z_n} \}_{n \in \omega}$. The first round of the
simple Borel strategy $\tau$ is given by $(A_n,y_n)$ where
$A_n =\{ (x_0,\rho_0)\colon x_0 \in I_{z_n}\setminus \bigcup_{m<n} I_{z_m} \}$ and
$y_n=\Sigma( z_n,\rho_0)$. Clearly $(A_n,y_n)$ is a simple one-round Borel strategy
which follows the rules $R$ of the $(\alpha,\beta,\rho)$ Schmidt's game.
This defines the first round of $\tau$. 
Using $\dc$, we continue inductively to define each subsequent round
of $\tau$ in a similar manner.

To see that $\tau$ is a winning strategy for II, simply note that for any run
of $\tau$ following the rules there is a run of $\Sigma$ producing the same
point of intersection. 
\end{proof}

This theorem immediately implies the following corollary about Schmidt's original $(\alpha, \beta)$ game.

\begin{corollary}[$\ad$]
For any $\alpha,\beta \in (0,1)$, and any $T\subseteq \R$,
exactly one of the following holds.
\begin{enumerate}
\item Player I has a winning strategy in Schmidt's $(\alpha,\beta)$ game.
\item For every $\rho \in \R_{>0}$, player II has a winning strategy in Schmidt's $(\alpha,\beta, \rho)$ game.
\end{enumerate}
\end{corollary}

In contrast to these results, the situation is dramatically different for $\R^n$, $n \geq 3$.

\begin{theorem} \label{thm:r3}
$\ad^+$ does not imply that the $(\alpha, \beta, \rho)$ Schmidt's game for $T \subseteq \R^n$, $n \geq 3$
are determined.
\end{theorem}

\begin{proof}
We will show that the determinacy of these games in $\R^3$ implies that
all relations $R \subseteq \mathbb{R} \times \mathbb{R}$ can be
uniformized.   It is known that $\ad^+$ does not suffice to imply this.  The proof for larger $n$ is identical. 

Let $R \subseteq \mathbb{R} \times [0, 2\pi)$ such that $\forall x \in
\mathbb{R}~ \exists \theta \in [0, 2\pi)~ (x, \theta) \in R$.
Let $r=\rho - 2\rho\alpha(1-\beta) \sum_{n=0}^\infty (\alpha \beta)^n$.
Let the target set for player II be $T=\{ (x, r\cos\theta, r\sin\theta) \colon (x, \theta) \in R\} \cup \{
(x, y, z) \colon  y^2+z^2>r\}$. The value $r$ is the distance from the
$x$-axis that is obtained if I makes a first move $B((x_0,0,0),\rho)$ centered on the $x$-axis,
and at each subsequent turn II moves to maximize the distance from the
$x$-axis and I moves to minimize it (note that these moves all have centers
having the same $x$-coordinate $x_0$). The target set $T$ codes
the relation $R$ to be uniformized along the boundary of the cylinder
of radius $r$ centered along the $x$-axis. 

We claim that I cannot win the $(\alpha, \beta, \rho)$ Schmidt's game for $T$.
First note that if I plays his center not on the $x$-axis,
then II can easily win in finitely many moves by simply playing to
maximize distance to the $x$-axis, (This will win the game by the
definition of $r$).  So suppose I plays $(x, 0, 0)$ as the center of
his first move. Fix $\theta$ so that $R(x, \theta)$ holds.  Then II
can win by always playing tangent towards the direction $(0,
\cos\theta, \sin\theta)$ maximizing distance to the $x$-axis.  If I
resists and minimizes distance to the $x$-axis, then the limit point
will be in $\{ (x, r\cos\theta, r\sin\theta) \colon (x, \theta) \in R\}$.
If I ever deviates from this, then again II can win after
finitely many moves by maximizing distance to the $x$-axis.

This shows that I does not have a winning strategy, so by the
assumption that these games are determined, II has a winning
strategy $\tau$.  By similar arguments to those above, $\tau$ must
maximize distance from the $x$-axis in response to optimal play by
I.  But one can take advantage of this to easily define a
uniformization $f$ of $R$ from $\tau$ by the following:

\[f(x) = \theta \Longleftrightarrow \tau\bigg(B\Big(\left(x, 0,0 \right), \rho\Big)\bigg)
= B\Big(\left(x, (\rho-\alpha\rho)\cos\theta, (\rho-\alpha\rho)\sin\theta\right), \alpha \rho\Big).\]
\end{proof}

\section{Further Results regarding Schmidt's game} \label{sec:or}

In \S\ref{sec:mr} we showed that $\ad$ suffices to get the determinacy of the $(\alpha,\beta,\rho)$
Schmidt's game for any target set $T\subseteq \R$, but that for $T\subseteq \R^n$, $n \geq 3$,
$\ad$ (or $\ad^+$) is not sufficient. The proof for the positive result in $\R$ used a reduction
of Schmidt's $(\alpha,\beta,\rho)$ game to a certain $\hr$ game. The fact that $\ad$ does not suffice
for $T\subseteq \R^n$, $n \geq 3$, shows that in general the $(\alpha,\beta,\rho)$ Schmidt's game is not
equivalent to an integer game (for $T\subseteq \R$ it still seems possible the game is equivalent to
an integer game).
A natural question is to what extent we can reduce Schmidt's game to an integer game.
In this section we prove two results concerning this question.

In the proof of Theorem~\ref{thm:r3} it is important that the value $r=r(\alpha,\beta)$ was calibrated to the
particular values of $\alpha$, $\beta$. In other words, if we change the values of $\alpha$, $\beta$ to $\alpha',\beta'$,
using the same target set, so that $r(\alpha'\beta')\neq r(\alpha,\beta)$, then the game is easily determined.
In Theorem~\ref{hyperbolathm} we prove a general result related to this phenomenon. Namely,
we show, assuming $\ad$, that for $T$ (in any Polish space) and each value of $p \in (0,1)$ there is at most
one value of $\alpha, \beta$ with $\alpha\beta=p$ such that the $(\alpha,\beta)$ Schmidt's game with target
set $T$ is not determined. Thus the values of $\alpha,\beta$ must be tuned precisely to have a possibility
of the game being not determined from $\ad$.

The proof of Theorem~\ref{thm:r3} also uses critically the ability of each player to play a ball tangent
to the previous ball. In Theorem~\ref{thm:tangent} below, we make this precise by showing
that the modification of Schmidt's $(\alpha, \beta, \rho)$ game where the players are required to make non-tangent moves is
determined from $\ad$ alone. Thus, the ability of the players to play tangent at each move is a key
obstacle in reducing Schmidt's game to an integer game.

In the Banach-Mazur game, the rational modification of the game is
fairly straightforward, i.e. the allowed moves for the players are
just representatives of balls with centers from some fixed countable
dense subset of $X$ and the radii are positive rationals, in Schmidt's
game there is a slight difference, again due to the restriction on the
players' radii.

\begin{definition}
For a Polish  $(X, d)$ and a fixed countable dense
subset $D \subseteq X$ we define the \emph{rational Schmidt} $(\alpha,
\beta)$ game by modifying Schmidt's
$(\alpha, \beta)$-game by restricting the set of allowed moves for
both players to balls $B(x_i,\rho_i)$ where $x_i \in D$ and
$\rho_i \in \left( \bigcup_{n, m \in
  \mathbb{N}} \alpha^n\beta^m\mathbb{Q}_{>0}\right)$. 
\end{definition}

\begin{theorem}
\label{lemma1}
Let $(X, d)$ be a Polish space.  Let $0<\alpha<\alpha'<1$,
$0<\beta'<\beta<1$, and $\alpha\beta=\alpha'\beta'$. Let $D$ be a
countable dense subset of $X$.
\begin{enumerate}
\item If II wins the rational Schmidt's $(\alpha', \beta')$ game for target set $T$
  then II wins Schmidt's $(\alpha, \beta)$ game for $T$. 
\item If I wins the rational Schmidt's $(\alpha, \beta)$ game for target set $T$
  then I wins Schmidt's $(\alpha', \beta')$ game for $T$.
\end{enumerate}
\end{theorem}

\begin{proof}
We will prove the first statement, the proof of the second is similar.
Fix the target set $T\subseteq X$.  Let $\tau$ be a winning strategy for II in the
rational Schmidt's $(\alpha', \beta')$ game.  We will construct a
strategy for II in Schmidt's $(\alpha, \beta)$ game by using
$\tau$.

Suppose I plays $(x_0, \rho_0)$ as his first move in the $(\alpha,\beta)$ game. Let $\rho=\rho_0$
to conserve notation. Let
$\rho' \in \left( \bigcup_{n, m \in \mathbb{N}}
\alpha^n\beta^m\mathbb{Q}_{>0}\right)$ with

\begin{equation}
\label{rhoprime}
\rho
\frac{\alpha}{\alpha'}\frac{1-\beta}{1-\beta'}<\rho'< \rho\frac{1-\alpha}{1-\alpha'}
\end{equation}
This is possible since
$\frac{\alpha}{\alpha'}\frac{1-\beta}{1-\beta'}<1$ and
$\frac{1-\alpha}{1-\alpha'}>1$ and $\bigcup_{n, m \in
  \mathbb{N}} \alpha^n\beta^m\mathbb{Q}_{>0}$ is dense in $\mathbb{R}^{>0}$.

Let $\epsilon_n\df\min\left\{(\alpha\beta)^n(\rho(1-\alpha)-\rho'(1-\alpha')),
(\alpha\beta)^{n-1}(\alpha'\rho'(1-\beta')-\alpha\rho(1-\beta))\right\}$.

Notice that $\epsilon_n >0$ by inequality (\ref{rhoprime}).  Now let
$(x_1', \alpha'\rho')=\tau(x_0', \rho')$ where $x_0' \in D \cap B(x_0, \epsilon_0)$.
Let $x_1=x_1'$.  By the
definition of $\epsilon_0$ and (\ref{rhoprime}), $B(x_1, \alpha\rho)\subseteq B(x_0,
\rho)$, thus $(x_1, \alpha\rho)$ is a valid response to $(x_0, \rho)$
in Schmidt's $(\alpha, \beta)$ game.

Now given a partial play with centers $\left\{x_k : k \leq
2n\right\}$, continue by induction to generate $x_{2n+1}$ by
considering $(x_{2n+1}', (\alpha'\beta')^n \alpha'\rho')=
\tau\left( \left\{(x_k', r_k) \colon k \leq
2n\right\}\right)$ where
for each $1\leq k\leq n$, $x_{2k-1}'$ is given by $\tau$ and
$x_{2k}'\in D\cap B(x_{2k}, \epsilon_k)$.  Again by
the definition of $\epsilon_n$ and (\ref{rhoprime}), $B(x_{2n+1},
(\alpha\beta)^n \alpha\rho) \subseteq  B(x_{2n}, (\alpha\beta)^n\rho)$.

We have defined a strategy for II in Schmidt's $(\alpha, \beta)$ game
which has the property that if a run is compatible with this
strategy with centers $\left\{x_k \colon k \in \omega\right\}$ then there
is a corresponding run compatible with $\tau$ with centers
$\left\{x_k' \colon k \in \omega\right\}$ such that for all $k$,
$x_{2k+1}=x_{2k+1}'$, so that $ \lim_{n \rightarrow
  \infty} x_n' = \lim_{n \rightarrow \infty} x_n$ and so since
$\tau$ is a winning strategy in the rational Schmidt's $(\alpha',
\beta')$ game, $\lim_{n \rightarrow \infty} x_n \in
T$.  So the strategy we have constructed is winning in Schmidt's
$(\alpha, \beta)$ game.
\end{proof}

As a consequence we have the following theorem.

\begin{theorem}[$\ad$]
\label{hyperbolathm}  Let $(X, d)$ be a Polish space. Let $T \subseteq
X$. Let $p \in (0, 1)$, then there is at most one point $(\alpha,
\beta) \in (0, 1)^2$ with $\alpha\beta=p$ at which Schmidt's $(\alpha,
\beta)$ game for $T$ is not determined.
\end{theorem}

\begin{proof}
Suppose that Schmidt's $(\alpha, \beta)$ game is not determined
with $\alpha\beta=p$. Let $\alpha_1 <\alpha <\alpha_2$ and $\beta_1 >
\beta > \beta_2$ with $\alpha_1 \beta_1 = \alpha \beta=\alpha_2
\beta_2$.  Note that by Theorem~\ref{lemma1} part (1), II cannot have
a winning strategy in the rational Schmidt's $(\alpha_2, \beta_2)$ game,
since II does not have a winning strategy in Schmidt's
$(\alpha, \beta)$ game by assumption.  This means that I must
have a winning strategy in the rational Schmidt's $(\alpha_2, \beta_2)$ game
for any such $(\alpha_2, \beta_2)$ (by $\ad$) and thus by
Theorem~\ref{lemma1} part (2), I wins Schmidt's $(\gamma, \delta)$ game
for any $(\gamma, \delta) \in (0, 1)^2$ with $\gamma\delta=p$ and
$\alpha<\gamma$.  By a symmetric argument, I has no winning strategy
in the rational Schmidt's $(\alpha_1, \beta_1)$ game, so II must
have a winning strategy in Schmidt's $(\gamma, \delta)$ game for
any $(\gamma, \delta) \in (0, 1)^2$ with $\gamma\delta=p$ and
$\gamma<\alpha$.
\end{proof}

We next consider the variation of Schmidt's game where we restrict the
players to making non-tangent moves. We consider a general Polish space $(X,d)$.

\begin{definition}
We say the ball $B(x_{n+1}, \rho_{n+1})$ is \emph{tangent} to the ball $B(x_n, \rho_n)$ if
$\rho_{n+1} + d(x_n, x_{n+1}) = \rho_n$.
\end{definition}

In the {\em non-tangent} Schmidt's $(\alpha,\beta,\rho)$ game with target set $T \subseteq X$,
a rule of the game is that each player must play a nested ball of the appropriate radius, as in Schmidt's
game, but that ball must not be tangent to the previous ball. Note that the non-tangent variation
of Schmidt's game is still an intersection game, and the rule set $R$ is still Borel.
We will show that the ``simple strategy'' condition of Theorem~\ref{detthm} is also
satisfied, and so the non-tangent Schmidt's game is determined from $\ad$. The proof of this theorem is 
similar to that of Theorem~\ref{thm:schmidtdet}.  It is clear that the rules of this game are positional, so it will suffice to check the other hypotheses of Theorem~\ref{detthm}.

\begin{theorem}[$\ad$]\label{thm:tangent}
Let $(X,d)$ be a Polish space, and let $\alpha,\beta \in (0,1)$, $\rho \in \R_{>0}$, and $T\subseteq X$,
the non-tangent $(\alpha,\beta,\rho)$ Schmidt's game with target set $T$ is determined. 
\end{theorem}

\begin{proof}
We will show that if I (or II) has a winning strategy in the non-tangent $(\alpha,\beta,\rho)$ Schmidt
game, then I (or II) has a simple Borel winning strategy (in the sense of
Definition~\ref{simple_one_round}), thus by Theorem~\ref{detthm}, the result follows.

Without loss of generality, say II has a winning strategy $\Sigma$ in the non-tangent
$(\alpha,\beta,\rho)$ Schmidt's game.  We will define a simple Borel strategy $\tau$ for II from $\Sigma$. 
Suppose I makes first move $B(x_0,\rho)$,
and $\Sigma$ responds with $B(x_1,\alpha\rho)$, which is not tangent to $B(x_0,\rho)$.
Let $\epsilon= \rho(1-\alpha)-d(x_0,x_1)>0$.  If $d(x'_0,x_0) <\epsilon$,
then if I plays $B(x'_0,\rho)$, then $B(x_1,\alpha\rho)$ is still a valid response for 
for II. In other words, for each $x_0$, there is an open ball $U$ of some radius, for which
any $x'_0 \in U$ has the property that the response by $\Sigma$ to $(x_0, \rho)$ is also a legal
response to $(x'_0, \rho)$.  Let $\mathcal{C}$ be the collection of all such open balls $U$.
Then $\mathcal{C}$ is an open cover of $X$, and since $X$ is Polish, it is Lindel\"of, and thus
$\mathcal{C}$ has a countable subcover $\mathcal{C}' = \setof{U_{z_n}}_{n \in \omega}$.  The first round of the
simple Borel strategy $\tau$ is given by $(A_n,y_n)$ where
$A_n =\{ (x_0,\rho)\colon x_0 \in U_{z_n}\setminus \bigcup_{m<n} U_{z_m} \}$ and
$y_n=\Sigma(z_n,\rho)$. Clearly $(A_n,y_n)$ is a simple one-round Borel strategy
which follows the rules $R$ of the non-tangent $(\alpha,\beta,\rho)$ Schmidt's game.
This defines the first round of $\tau$. 
Using $\dc$, we continue inductively to define each subsequent round
of $\tau$ in a similar manner.

To see that $\tau$ is a winning strategy for II, simply note that for any run
of $\tau$ following the rules there is a run of $\Sigma$ producing the same
point of intersection. 

\end{proof}

\section{Questions} \label{sec:questions}
In Theorem~\ref{thm:schmidtdet} we showed that $\ad$ suffices to the determinacy
of Schmidt's $(\alpha,\beta,\rho)$ game on $\R$. In Theorem~\ref{thm:r3}
we showed that $\adp$ does not suffice to prove the determinacy of
Schmidt's $(\alpha,\beta,\rho)$ game on $\R^n$ for $n \geq 3$. In view of these results
several natural questions arise.

First, for $n=2$ our arguments do not seem to resolve the question of the strength
of Schmidt game determinacy in either case of the $(\alpha,\beta,\rho)$ or the
$(\alpha,\beta)$ game. The proof of Theorem~\ref{thm:schmidtdet} does not immediately
apply as $\R^2$ does not have the ``Lindel\"{o}f-like'' property we used for $\R$.
On the other hand, the proof of Theorem~\ref{thm:r3} also does not seem to apply
as we don't seem to have enough freedom in $\R^2$ to code an arbitrary
instance of uniformization as we did in $\R^3$.
In fact, the method of proof of Theorem~\ref{thm:schmidtdet} of using ``simple strategies''
cannot show the determinacy of Schmidt games in $\R^2$ from $\ad$. This is because
while we cannot seem to code an arbitrary uniformization problem into the game,
we can code the characteristic function of an arbitrary set $A\subseteq \R$
in a way similar to the proof of Theorem~\ref{thm:r3}. We could then choose
a set $A$ not projective over the pointclass $\bG$ (as in the statement of
Theorem~\ref{detthm}). Then the ``simple strategy'' hypothesis of
Theorem~\ref{detthm} will fail for this instance of the game.
So we ask:

\begin{question} \label{qa}
Does $\ad$ suffice to get the determinacy of either the Schmidt's $(\alpha,\beta,\rho)$
or $(\alpha,\beta)$ games on $\R^2$?
\end{question}

Although the distinction between Schmidt's $(\alpha,\beta,\rho)$ game and
Schmidt's $(\alpha,\beta)$ game seem immaterial in practical applications,
our main theorems apply to the $(\alpha,\beta,\rho)$ games only. So we ask:

\begin{question} \label{qb}
Does $\ad$ suffice to prove the determinacy of Schmidt's $(\alpha,\beta)$
game on $\R^n$?
\end{question}

Also interesting is the converse question of whether the determinacy of
Schmidt's game (either variation) implies determinacy axioms. In
\cite{Freiling} it is shown that the determinacy of Banach games
(which are similar in spirit to Schmidt games) implies $\ad$.
Here we do not have a corresponding result for $\R^n$. We note though
that if $\alpha=\beta=\frac{1}{2}$ and $\rho=\frac{1}{2}$, then the
determinacy of Schmidt's $(\alpha, \beta, \rho)$ game on $X=\ww$ with the standard metric
$d(x,y)= \frac{1}{2^{n+1}}$ where $n$ is least so that $x(n)\neq y(n)$,
gives $\ad$. So we ask:

\begin{question} \label{qc}
Does the determinacy of Schmidt's $(\alpha,\beta,\rho)$ (or $(\alpha,\beta)$)
game on $\R^n$ imply $\ad$? If $n \geq 3$, does Schmidt determinacy
imply $\adr$?
\end{question}

A related line of questioning is to ask what hypotheses are needed to get the determincy of
Schmidt's game for restricted classes of target sets. For example, while the determinacy of
the Banach-Mazur game for $\bS^1_1$ (that is, analytic) target sets is a theorem of just $\zf$, the corresponding
situation for Schmidt's game is not clear.  so we ask:

\begin{question} \label{qd}
Does $\zf+\dc$ suffice to prove the determinacy of Schmidt's game in $\R^n$
for $\bS^1_1$ target sets?
\end{question}

In view of the results of this paper, it is possible that the answer to Question~\ref{qd}
depends on $n$.  We can extend the class of target sets from the analytic sets
to the more general class of Suslin, co-Suslin sets. So we ask:

\begin{question} \label{qe}
Does $\ad$ suffice to prove the determinacy of Schmidt's game in $\R^n$
for Suslin, co-Suslin target sets?
\end{question}

Again, it is possible the answer to Question~\ref{qe} depends on $n$.

Finally, it is reasonable to ask the same questions of this paper for other
real games which also have practical application to number theory and related
aread. Important examples include McMullen's ``strong'' and ``absolute'' variations of
Schmidt's game \cite{McMullen_absolute_winning}. These are also clearly intersection games,
so the question is whether the simple strategy hypothesis of
Theorem~\ref{detthm} applies.

\bibliographystyle{amsplain}

\bibliography{Schmidt}

\end{document}